\documentclass[10pt,reqno]{amsart}

\usepackage{hyperref}
\usepackage{color}
\usepackage{amssymb}

\usepackage{graphicx, comment}
\usepackage{enumerate}
\usepackage{ytableau}
\usepackage{comment}
\usepackage{amsmath}
\usepackage{mathtools}
\usepackage{graphicx}
\usepackage{tikz}
\usetikzlibrary{matrix,positioning}

\newtheorem{Theorem}{Theorem}
\newtheorem{Proposition}[Theorem]{Proposition}

\newtheorem{Lemma}[Theorem]{Lemma}

\theoremstyle{remark}
\newtheorem*{Remark}{Remark}

\newtheorem*{rems}{Remarks} 
\newenvironment{Remarks}{\begin{rems}\normalfont}{\end{rems}}

\numberwithin{equation}{section}

\allowdisplaybreaks 

\newcommand{\pqrfac}[3]{{\left({#1};#3\right)_{#2}}}

\newcommand{\elliptictheta}[1]{\theta\!\left({#1} ; p\right) }

\newcommand{\ellipticptheta}[2]{\theta\!\left({#1} ; #2\right) }

\newcommand{\ellipticqrfac}[2]{{\left({#1}; q, p\right)_{#2}}} 
\newcommand{\ellipticpqrfac}[4]{{\left({#1}; #3, #4\right)_{#2}}} 

\newcommand{\D}{{ \mathbf D}}

\newcommand{\B}{{ \mathbf B}}
\newcommand{\F}{{\mathbf F}}
\newcommand{\G}{{\mathbf G}}
\newcommand{\bH}{{\mathbf H}}
\newcommand{\M}{{\mathbf M}}

\newcommand{\bbeta}{{\boldsymbol \beta}}
\newcommand{\balpha}{{\boldsymbol \alpha}}

\newcommand{\qrfac}[2]{{\big({#1}; q\big)_{#2}}} 

\author[G.~Bhatnagar]{Gaurav Bhatnagar
}

\address{Ashoka University, Sonipat, Haryana 131029, India}
\email{bhatnagarg@gmail.com}

\author[A.~Kumari]{Archna Kumari}
\address{Department of Mathematics, Indian Institute of Technology, Delhi 110067, India.}
\email{arcyadav856@gmail.com}

\title{Expansion formulas for elliptic hypergeometric series}
 


\begin{document}

\begin{abstract}
We provide an alternate approach to obtaining expansion formulas on the lines of the well-poised Bailey lemma. We recover results due to Spiridonov and Warnaar and one new formula of this type. 
These formulas contain an arbitrary sequence as an argument, and are thus flexible in the number of parameters they contain. As a result, we are able to derive eight new transformation formulas for elliptic hypergeometric series. These transformation formulas appear to be new even in the basic hypergeometric case, when $p=0$.
\end{abstract}

\maketitle

\section{Introduction}

There are a large number of summation and transformation formulas for hypergeometric and basic hypergeometric series, that is, series of the form 
$\sum t_k$, where the ratio $t_{k+1}/t_k$ of successive terms is a rational function of the index $k$, or, a rational function of $q^k$. There are relatively fewer results for elliptic hypergeometric series, where the ratio of successive terms is an elliptic function, the index $k$ now regarded as a complex variable.  
The objective of this paper is to provide a promising general technique, which is quite productive in finding new transformation formulas, and use it to find eight such formulas for elliptic hypergeometric series. 

Among the first results for elliptic hypergeometric series are two fundamental formulas given by Frenkel and Turaev~\cite{FT1997}, one of 
which sums a $_{10}V_9$ sum, and the other, a transformation formula for $_{12}V_{11}$ series.  Subsequently, Spiridonov~\cite{VPS2002} and 
Warnaar~\cite{SOW2002, SOW2003} popularized the area, obtained many results, and proposed several conjectures.   Further summation and 
transformation formulas appear in~\cite{BKS2024, CJ2008, CJ2013, GS2005,  LSW2009, LRW2020, RS2005, RS2020, SY2016c, SOW2005,
ZD2016}. Gasper and Rahman~\cite[Chapter 11]{GR90} and, more recently, Rosengren~\cite{HR2016-lectures} have given accessible 
introductions. 


As regards general techniques in the study of elliptic hypergeometric series, Warnaar~\cite{SOW2002} used matrix inversion to obtain many transformation formulas; Spiridonov~\cite{VPS2002}  established the well-poised  (or simply, the WP) Bailey lemma, and Warnaar~\cite{SOW2003} gave several analogous results. In addition, Chu and Jia~\cite{CJ2008, CJ2013} gave a general approach which leads to several identities. We offer a mild modification of the WP Bailey lemma---a slight change of perspective---that leads to several new transformation formulas.

Like the WP Bailey lemma, our approach requires a pair of inverse, infinite, lower-triangular, matrices. 
Thus, we require a pair of matrices
$\F=(F_{km})$ and $\G=(G_{km})$, with rows and columns indexed by non-negative integers, such that
$F_{km} =0 = G_{km}$  for  $m > k$,
and 
$$\sum_{j=k}^m F_{kj}G_{jm} =\delta_{km} = \sum_{j=k}^m G_{kj}F_{jm}.$$
Here $\delta_{km}$ is the delta function, which is $1$ when $k=m$, and $0$, otherwise. 

A key role is played by the entries of a matrix $\M$ given by:
\begin{equation}\label{M-terms}
M_{km}(a, b,  q, p) =  
\frac{ \ellipticqrfac{ aq^k, q^{-k}}{m}}{\ellipticqrfac{bq^{1-k}/a, bq^{1+k}}{m}},
\end{equation}
where the $q,p$-shifted factorial $\ellipticqrfac{a}{m}$ and other notations used here are defined shortly. 

If a balanced, well-poised sum, which terminates naturally after $k+1$ terms, has an additional parameter $a$, then it must contain this factor. Here the index of summation is $m$. 
This is  due to a classical theorem on elliptic functions (see Rosengren~\cite[Theorem~4.12, p.~221]{HR2016-lectures}), which forces 
the terms of an elliptic hypergeometric series to satisfy a balancing condition.
The balancing condition ensures that there is a numerator parameter of the form $aq^k$. Further, the well-poised condition implies a matching parameter in the denominator, of the form $bq^{1+k}$. The well-poised condition is satisfied because: 
$$aq^k \cdot bq^{1-k}/a = q^{-k}\cdot bq^{1+k} = bq.$$
Thus the factors in \eqref{M-terms} are present in the summand of any terminating, well-poised summation theorem for elliptic hypergeometric series. 

An essential ingredient of our technique is the following finite version of a lemma from Bhatnagar and Rai~\cite[Lemma~5.1]{BR2022}.
\begin{Lemma}\label{lemma:qLidea} Let $\F=(F_{km})$, $\G=(G_{km})$ and $\bH =(H_{km})$ be infinite, lower-triangular matrices, with rows and columns indexed by non-negative integers; let $\F$ and $\G$ be inverses of each other. Let ${\balpha}=(\alpha_j)$ be a sequence. Then
\begin{equation}\label{lemma-idea}
\sum_{j=0}^n H_{nj}  \alpha_j
 = 
 \sum_{k=0}^n F_{nk}  \sum_{j=0}^k \alpha_j \sum_{m=0}^{k-j} G_{k,j+m} H_{j+m,j}.
\end{equation}
\end{Lemma}
The lemma follows by interchanging the summations on the right-hand side (twice). Alternatively, it follows from the matrix equalities
$$\bH\balpha = (\F\G)\bH\balpha = \F (\G\bH)\balpha.$$

Lemma~\ref{lemma:qLidea} is motivated by work on $q$-Lagrange inversion and its connection with matrix inversion; this is work of Gessel and Stanton~\cite{GS1983}, Andrews, Agarwal and Bressoud~\cite{AAB1988}, and others \cite{BR2022, DB1988, WM2017, SOW2003}. 

The key idea of this paper is to choose $\F$ and $\bH$ so that 
the inner sum on the right-hand side is summable by a known summation theorem. We call the resulting formula an expansion formula, motivated by Gasper and Rahman's description of \cite[Equation~(2.2.2)]{GR90}. This still leaves an arbitrary sequence $(\alpha_j)$ as a parameter. On further specialization, and application of a (possibly different) summation theorem, it yields a transformation formula.
What makes the application of this approach particularly simple, is that the entries of $\M$ appear---in some form---in all of $\F$, $\G$, 
$\bH$ and $\G\bH$.

Lemma~\ref{lemma:qLidea} is related to the WP Bailey lemma. The WP Bailey lemma uses a matrix (see \eqref{Bressoud}, below) denoted by
$\B=(B_{km})$. This matrix was given by Bressoud~\cite{DB1983} in the $q$-hypergeometric case. 
Let $\balpha$ and $\bbeta$ be two sequences connected by 
$$\bbeta = \B \balpha.$$
The pair $(\balpha, \bbeta)$ that satisfies this relationship is called a well-poised Bailey pair.
The WP Bailey lemma provides another well-poised Bailey pair $(\balpha^\prime, \bbeta^\prime)$. 
The WP Bailey chain is obtained by iteration; the lattice is obtained by changing one of the parameters; the Bailey tree (and various branches) appear by using analogous results. See Andrews and Berkovich~\cite{AB2002} and Andrews~\cite{Andrews2001}. 

Now, in Lemma~\ref{lemma:qLidea}, we choose 
the matrix $\F$ to be Warnaar's~\cite{SOW2002} elliptic extension of Bressoud's  matrix.  When we further choose $\bH$ so that the inner sum can be summed using Frenkel and Turaev's \cite{FT1997}  elliptic summation formula, for a terminating, balanced, and very-well-poised $_{10}V_9$, we obtain an expansion formula. A referee pointed out that this formula is equivalent to the elliptic WP Bailey lemma due to Spiridonov~\cite[Theorem~4.3]{VPS2002}. The left-hand side of \eqref{lemma-idea} corresponds to $\bbeta^\prime$ $(=\B \balpha^\prime)$ written out explicitly in terms of $\balpha$;  the right-hand side matches too, when written in terms of $\balpha$.  

Similarly, given a summation theorem, we may choose $\bH$ appropriately to ensure that the inner sum is summable by this summation, obtaining a result analogous to the WP Bailey lemma. This is what we do here---with four known summation theorems---keeping $\F$ to be the aforementioned matrix of Warnaar. In addition to Frenkel and Turaev's sum, we use two summation theorems due to 
Warnaar~\cite{SOW2005}, and one from Lee, Rains and Warnaar~\cite{LRW2020}. These four summations are used to recover WP Bailey lemma type results, including Spiridonov's WP Bailey lemma, as well as two of Warnaar's results~\cite{SOW2003}. In addition, we obtain one new Bailey lemma type result. These are the expansion formulas in 
\S\ref{sec:expansions}  of this paper. 

Next, we apply these four summation theorems, and two more due to Warnaar~\cite{SOW2005}, to these four expansion formulas. Of the 24 transformation formulas thus obtained, eight are new; these are listed in \S\ref{sec:transformations}. We conclude in  \S\ref{sec:special} with some remarks on basic hypergeometric series obtained when the nome $p=0$. In particular, we highlight a transformation formula that appears to be a close cousin of a formula of Bailey~\cite[Equation (2.8.3)]{GR90}, and obtain two summation theorems as special cases.


We now 
list some notation used in this paper. For further background information, we recommend
Gasper and Rahman~\cite[Chapter 11]{GR90} and Rosengren~\cite{HR2016-lectures}.
\subsection*{Notation} \ 
\begin{enumerate}
\item We define the {\em $q$-shifted factorials}, for $k$ a non-negative integer, as
\begin{gather*}
\qrfac{a}{k} :=\prod\limits_{j=0}^{k-1} \big(1-aq^j\big),
\end{gather*}
and for $|q|<1$,
\begin{gather*}
\qrfac{a}{\infty} := \prod\limits_{j=0}^{\infty} \big(1-aq^j\big).
\end{gather*}
The parameter $q$ is called the {\em base}. With this definition, we have the {\em modified Jacobi theta function}, defined as
\begin{gather*} \elliptictheta{a} := \pqrfac{a}{\infty}{p} \pqrfac{p/a}{\infty}{p},\end{gather*}
where $a\neq 0$ is a complex number, and $|p|<1$. We define the {\em $q, p$-shifted factorials} (or {\em theta shifted factorials}), for $k$ a non-negative integer, as
\begin{gather*}
\ellipticqrfac{a}{k} := \prod\limits_{j=0}^{k-1} \elliptictheta{aq^j}.
\end{gather*}
The parameter $p$ is called the {\em nome}. When the nome $p=0$, the modified theta function $\elliptictheta{a}$ reduces to $(1-a)$; and $\ellipticqrfac{a}{k}$ reduces to $ \pqrfac{a}{k}{q}$.
\item We use the shorthand notations:
\begin{gather*}
\elliptictheta{a_1, a_2, \dots, a_r} := \elliptictheta{a_1} \elliptictheta{a_2}\cdots \elliptictheta{a_r};\\
\ellipticqrfac{a_1, a_2,\dots, a_r}{k} := \ellipticqrfac{a_1}{k} \ellipticqrfac{a_2}{k}\cdots \ellipticqrfac{a_r}{k};\\
\qrfac{a_1, a_2,\dots, a_r}{k} := \qrfac{a_1}{k} \qrfac{a_2}{k}\cdots \qrfac{a_r}{k}.
\end{gather*}

\item Note that for non-negative integers $n$ and $k$:
\begin{equation}\label{terminating}
\qrfac{q^{-n}}{k} = 0 = \ellipticqrfac{q^{-n}}{k} \text{ if } k > n.
\end{equation}

\item The terms {\em balanced} and  {\em very-well-poised} are used to describe the relationship of the parameters for elliptic hypergeometric series. In addition, there is a notation $${}_{r+1}V_r(a_1;a_6,\dots,a_{r+1};q,p)$$ which is used to describe such series with $r+1$ numerator parameters. We refer to \cite{GR90, HR2016-lectures} for their definitions. In what follows, we use these descriptive terms, as well as the corresponding terminology for basic hypergeometric series ($_r\phi_s$  series). 
\end{enumerate}
We use elementary identities from \cite[Ch.\ 11]{GR90} without comment.

\section{Expansion formulas}\label{sec:expansions} 

In this section, we derive four expansion formulas by using Lemma~\ref{lemma:qLidea}. 
The pair of inverse matrices we use for this derivation are an equivalent form of inverse matrices given by Warnaar~\cite{SOW2002}. 
\begin{Proposition}[Warnaar]\label{warnaar-matrix}
Let $\F=\big(F_{km}(a,b)\big)$ and $\G=\big(G_{km}(a,b))$ be defined as
\begin{subequations}
\begin{align}
F_{km}(a,b) &=  
\frac{\elliptictheta{aq^{2m}} \ellipticqrfac{bq^k, q^{-k}}{m}\ellipticqrfac{b/a}{k}}{\ellipticqrfac{aq^{1+k}, aq^{1-k}/b}{m}}
\Big(\!-\frac{a}{b}\Big)^m q^{\binom {m+1}2}; \label{matrix-F}\\
G_{km}(a,b) & = \frac{\elliptictheta{bq^{2m}}\ellipticqrfac{aq^{m+1}}{k-1}\ellipticqrfac{a/b}{k}
\ellipticqrfac{q^{-k}}{m}}{\ellipticqrfac{bq^{1-k}/a, q}{m} \ellipticqrfac{q}{k}\ellipticqrfac{bq^m}{k+1}} 
\Big(\!-\frac{b}{a}\Big)^k q^{m-\binom k2 }.\label{matrix-G}
\end{align}
\end{subequations}
Then $\F$ and $\G$ are inverses of each other.
\end{Proposition}
It is easy to see  that the matrix inversion of $\F$ and $\G$ is equivalent to the matrix inversion of $\B=\big( B_{km}\big)$ and its inverse, with entries:
\begin{subequations}
\begin{align}
B_{km} &=  \frac{\ellipticqrfac{b}{k+m} \ellipticqrfac{b/a}{k-m}} { \ellipticqrfac{aq}{k+m}\ellipticqrfac{q}{k-m}}, \label{Bressoud}
\intertext{and}
B^{-1}_{k m} &=\frac{\elliptictheta{aq^{2k}}}{\elliptictheta{a}}\frac{\elliptictheta{bq^{2m}}}{\elliptictheta{b}}\frac{\ellipticqrfac{a}{m+k} \ellipticqrfac{a/b}{k-m}}
{ \ellipticqrfac{bq}{m+k}\ellipticqrfac{q}{k-m}}\left( \frac{b}{a}\right)^{k-m}.
\end{align}
\end{subequations}
This is the Warnaar's matrix inversion as written by Bhatnagar and Schlosser~\cite[Cor.~4.5]{BS2018a} (with $n=1$, $x_1=1$).

\begin{Remark} Note that the matrix $\M$ given in \eqref{M-terms} is essentially the same as the matrix $\B$. Observe that $M_{km}(b,a,q,p)$ can be obtained from $B_{km}$ by multiplying by suitable diagonal matrices.
\end{Remark}

In our derivations below, we find it convenient to take $h_j(m)$ and $K(m)$ such that
\begin{equation*}
H_{mj} = h_{j}(m)K(m).
\end{equation*}
This separates out factors of $H_{mj}$ that depend on both $m$ and $j$. We will always take $h_j(m)$ to be $M_{mj}(a,b,q,p)$ with some substitutions for $a, b, q, p$. 

\subsection*{First expansion formula}
For our first expansion formula, we use Proposition~\ref{warnaar-matrix} and
the  $_{10}V_9$ summation formula due to 
Frenkel and Turaev~\cite{FT1997}; see \cite[equation~(11.4.1)]{GR90}, which may be written as: 
\begin{gather}\label{10V9}
\sum_{j=0}^k
\frac
{\ellipticptheta{aq^{2j}}{p} \ellipticqrfac{a, b, c, d, a^2q^{k+1}/bcd, q^{-k}}{j}q^{j}}
{\ellipticptheta{a}{p} \ellipticqrfac{q, aq/b, aq/c, aq/d, bcdq^{-k}/c, aq^{k+1}}{j}}
=
\frac {\ellipticqrfac{aq,aq/bc,aq/bd,aq/cd}n} {\ellipticqrfac{aq/b,aq/c,aq/d,aq/bcd}n}.
\end{gather}

Our first expansion formula is equivalent to the well-poised Bailey lemma given by Spiridonov~\cite[Theorem~5.1]{VPS2002}. It is given by:
\begin{multline}\label{ell-2}
\frac
{\ellipticqrfac{aq, c/b, d,  bcq/ad}{n}}
{\ellipticqrfac{c/a, bq, cq/d, ad/b}{n}}
\sum_{j=0}^n
\frac{
\ellipticqrfac{ cq^n,q^{-n}, bq/d, ad/c}{j}}
{
\ellipticqrfac{
bq^{1-n}/c, bq^{1+n},  d, bcq/ad}{j}}
A_{j}\\
=\sum_{k=0}^n \frac{\elliptictheta{aq^{2k}}
\ellipticqrfac{a, q^{-n},cq^n, a/b, ad/c, bq/d}{k}}
{\elliptictheta{a}\ellipticqrfac{q, 
aq^{1+n},aq^{1-n}/c, bq, cq/d, ad/b}{k}}q^k
\\
\times 
\sum_{j=0}^k 
\frac{
\ellipticqrfac{aq^k, q^{-k}}{j}}
{
\ellipticqrfac{
bq^{1-k}/a, bq^{1+k}}{j}}
A_{j}.
\end{multline}

\begin{Remark}
We observe that the summand of the sum on the left-hand side of \eqref{ell-2} contains $M_{nj}(c, b,  q, p)$ (see \eqref{M-terms}), while the inner sum on the right-hand side contains
$M_{kj}(a, b,  q, p)$.
\end{Remark}

\begin{proof}[Proof of \eqref{ell-2}]
We take $F_{mj}(a,c)$ from \eqref{matrix-F} as the matrix $\F$ and the corresponding $ G_{mj}(a,c)$ in Lemma~\ref{lemma:qLidea}.
We take $H_{mj} = h_j(m) K(m)$, with
$h_{j}(m)= F_{mj}(b,c)$,  and,
\begin{equation*}
K(m)= \frac{\ellipticqrfac{d, bcq/ad, aq}{m}}{\ellipticqrfac{cq/d, ad/b, bq}{m}}
.
\end{equation*}
The inner sum on the right-hand side is summed using the
$$(a, b, c, d, n)\mapsto  (cq^{2j}, aq^{j+k}, c/b, dq^{j}, k-j) 
$$
case of \eqref{10V9}. 
After some elementary calculations, we adjust the final form of \eqref{ell-2} by taking
$$\alpha_{j}=(-1)^j\Big(\frac{c}{bq}\Big)^j  \frac{\elliptictheta{b}\ellipticqrfac{ad/c,bq/d}{j}}
{\elliptictheta{bq^{2j}}\ellipticqrfac{ d, bcq/ad}{j}}
q^{-\binom{j}{2}}A_{j}.$$
\end{proof}

\subsection*{Second expansion formula}
The second expansion formula is equivalent to Warnaar~\cite[Theorem 5.2]{SOW2003}:
\begin{multline}\label{expansion2}
\frac{\elliptictheta{-cq^{2n}}\ellipticqrfac{-c, aq}{n}  \ellipticpqrfac{c/aq}{n}{q^2}{p^2} }
{\elliptictheta{-c} \ellipticqrfac{-q, c/a }{n}\ellipticpqrfac{acq^3}{n}{q^2}{p^2}}q^n
\\
\times 
\sum_{j=0}^{n} 
\frac{  \ellipticpqrfac{c^2q^{2n}, q^{-2n} }{j}{q^2}{p^2} \ellipticqrfac{-aq}{2j}}
{ \ellipticpqrfac{aq^{3-2n}/c, acq^{3+2n} }{j}{q^2}{p^2} \ellipticqrfac{-cq}{2j} }
 \Big(\frac{a}{c}\Big)^j A_{j}\\
= \sum_{k=0}^{n} 
 \frac{ \ellipticptheta{a^2q^{4k}}{p^2}\ellipticqrfac{cq^{n}, q^{-n}}{k} \ellipticpqrfac{a^2, a/cq }{k}{q^2}{p^2}}
{ \ellipticptheta{a^2}{p^2}  \ellipticqrfac{ aq^{1-n}/c, aq^{1+n}}{k}
\ellipticpqrfac{q^2, acq^3 }{k}{q^2}{p^2}}q^{2k}
\sum_{j=0}^k \frac{\ellipticpqrfac{ a^2q^{2k}, q^{-2k} }{j}{q^2}{p^2} }
{\ellipticpqrfac{cq^{3-2k}/a, acq^{3+2k} }{j}{q^2}{p^2}  }   A_{j}.
\end{multline}
\begin{Remark} Notice the entries $M_{nj}(c^2, acq,  q^2, p^2)$ in the summand of the sum on the left-hand side of \eqref{expansion2}, and
$M_{kj}(a^2, acq,  q^2, p^2)$ on the right-hand side.
\end{Remark}

Our derivation requires Warnaar's summation \cite[Eq.~(1.3)]{SOW2005}, which we write as
\begin{multline}\label{SOW1.3}
\sum_{j=0}^k
\frac
{\ellipticptheta{a^2q^{4j}}{p^2} \ellipticpqrfac{a^2, b/q}{j}{q^2}{p^2}\ellipticqrfac{aq^k/b, q^{-k}}{j} }
{\ellipticptheta{a^2}{p^2} \ellipticpqrfac{q^2, a^2q^3/b}{j}{q^2}{p^2} \ellipticqrfac{bq^{1-k}, aq^{1+k}}{j} }
q^{2j} \\
= 
\frac
{\elliptictheta{-aq^{2k}/b} \ellipticpqrfac{1/bq}{k}{q^2}{p^2}\ellipticqrfac{-a/b, aq}{k} }
{\elliptictheta{-a/b} \ellipticpqrfac{a^2q^3/b}{k}{q^2}{p^2} \ellipticqrfac{-q, 1/b}{k} }
q^k.
\end{multline}

\begin{proof}[Proof of \eqref{expansion2}]
We take $F_{mj}(a,c)$ from \eqref{matrix-F} as the matrix $\F$ and the corresponding $ G_{mj}(a,c)$ in Lemma~\ref{lemma:qLidea}.
We take $H_{mj} = h_j(m) K(m)$, with
$h_{j}(m)= M_{mj}(c^2, acq,  q^2, p^2)$,  and,
\begin{equation*}
K(m)= 
\frac{\elliptictheta{-cq^{2m}} \ellipticqrfac{aq}{m}\ellipticqrfac{-c}{m}\ellipticpqrfac{c/aq }{m}{q^2}{p^2}}
{\ellipticqrfac{-q}{m}\ellipticpqrfac{acq^3}{m}{q^2}{p^2}}q^{m}.
\end{equation*}

We use a suitable value of $\alpha_j$ to adjust the final result to obtain \eqref{expansion2}. This proof is essentially how Warnaar derived his original result. 
\end{proof}

\subsection*{Third expansion formula}
For the third expansion formula, we use 
Warnaar's summation formula \cite[Eq.~(1.4)]{SOW2005}: 
\begin{multline}\label{SOW1.4}
\sum_{j=0}^k
\frac
{\ellipticptheta{abq^{4j}}{p^2} \ellipticpqrfac{ab, aq^2/b, a^2q^{2k}, q^{-2k}}{j}{q^2}{p^2}
\ellipticqrfac{b}{2j}}
{\ellipticptheta{ab}{p^2} \ellipticpqrfac{q^2,  b^2, bq^{2-2k}/a, abq^{2+2k}}{j}{q^2}{p^2} 
\ellipticqrfac{aq}{2j}}
\Big( \frac{bq}{a}\Big)^j\\
= 
\frac
{\elliptictheta{a} \ellipticpqrfac{abq^2}{k}{q^2}{p^2}\ellipticqrfac{-q, aq/b}{k} }
{\elliptictheta{aq^{2k}} \ellipticpqrfac{a/b}{k}{q^2}{p^2} \ellipticqrfac{a, -b}{k} }
q^{-k}.
\end{multline}

For the third expansion formula, we use the matrices $\F$ and $\G$ obtained by replacing $q$ by $q^2$ and $p$ by $p^2$  in Proposition~\ref{warnaar-matrix}. 
We use $H_{mj} = h_j(m) K(m)$, with
$h_{j}(m)= M_{mj}(c, c^2/aq^2,  q^2, p^2)$,  and
\begin{equation*}
K(m)=  \frac{\ellipticpqrfac{aq^2, aq^2/c}{m}{q^2}{p^2}\ellipticqrfac{c/\sqrt{a}}{2m}}{\ellipticpqrfac{c^2/a}{m}{q^2}{p^2} \ellipticqrfac{\sqrt{a}q}{2m}}\bigg(\frac{c}{aq}\bigg)^{m}.
\end{equation*}
Now summing the inner sum on the right-hand side by \eqref{SOW1.4}, and replacing $\alpha_j$ by a suitable multiple of $A_j$, we obtain the following expansion formula.
%
%
%
%
 Again, it is an equivalent form of a WP Bailey lemma type result, due to Warnaar~\cite[Theorem~3.2]{SOW2003}:
\begin{multline}\label{expansion3}
\frac{
\ellipticpqrfac{a^2q^2, a^2q^2/c}{n}{q^2}{p^2}
\ellipticpqrfac{-c/a}{2n}{q}{p} }
{
\ellipticpqrfac{c/a^2, c^2/a^2}{n}{q^2}{p^2}
\ellipticpqrfac{-aq}{2n}{q}{p}
} 
\bigg( \frac{c}{a^2q}\bigg)^n 
\sum_{j=0}^{n} 
\frac{ \ellipticpqrfac{ cq^{2n}, q^{-2n}}{j}{q^2}{p^2}}
{\ellipticpqrfac{ cq^{-2n}/a^2, c^2q^{2n}/a^2} {j}{q^2}{p^2}}
 A_{j}\\
= \sum_{k=0}^{n} 
\frac{
\ellipticptheta{aq^{2k}}{p}} 
{\ellipticptheta{a}{p} 
} 
\frac{
 \ellipticqrfac{a, a^2q/c}{k} 
 \ellipticpqrfac{cq^{2n}, q^{-2n} }{k}{q^2}{p^2}}
{ \ellipticqrfac{q, c/a}{k}  
\ellipticpqrfac{  a^2q^{2-2n}/c, a^2q^{2+2n}}{k}{q^2}{p^2}
}
q^{k}\\
 \times \sum_{j=0}^k \frac{\ellipticqrfac{ aq^{k}, q^{-k} }{j}}
{\ellipticpqrfac{cq^{-k}/a^2, cq^{k}/a}{j}{q}{p} 
} 
 A_{j}.
\end{multline}
\begin{Remark}
Notice the factor $M_{nj}(c, c^2/a^2,  q^2, p^2)$ on the left-hand side of \eqref{expansion3}, and $M_{kj}(a, c/a,  q, p)$ on the right-hand side.
\end{Remark}

\subsection*{Fourth expansion formula}
Our final expansion formula uses a summation theorem due to 
Lee, Rains and Warnaar~\cite{LRW2020}. In  \cite[Eq.\ (2.19a)]{LRW2020}, we take $a\mapsto a/p$, simplify, and then take $a\mapsto b$ and $b\mapsto a/b$, to rewrite their result as:
\begin{equation}\label{lrw-1}
\sum_{j=0}^k
\frac{\elliptictheta{bq^{2j}} \ellipticqrfac{b, q^{-k}, aq^{k},   aq/b}{j}\ellipticpqrfac{b^2/a}{2j}{q}{p^2}q^{j}}
{\elliptictheta{b} \ellipticqrfac{q, bq^{k+1}, bq^{1-k}/a,  b^2/a}{j} \ellipticpqrfac{aq}{2j}{q}{p^2}}
=\frac{\ellipticptheta{a}{p^2}\ellipticqrfac{bq}{k}
\ellipticpqrfac{pq, a^2q/b^2}{k}{q}{p^2}}
{\ellipticptheta{aq^{2k}}{p^2} \ellipticqrfac{a/b}{k}
\ellipticpqrfac{a, b^2p/a}{k}{q}{p^2}}.
\end{equation}

We use the matrices $\F$ and $\G$ from Proposition~\ref{warnaar-matrix}; $H_{mj} = h_j(m) K(m)$, with
$h_{j}(m)= M_{mj}(c, c^2/aq,  q, p)$,  and
\begin{equation*}
K(m)=  \frac{\ellipticqrfac{aq}{m}\ellipticqrfac{aq/c}{m}\ellipticpqrfac{c^2/a }{2m}{q}{p^2}}{\ellipticqrfac{c^2/a}{m}\ellipticpqrfac{aq}{2m}{q}{p^2}}.
\end{equation*}
Next, we sum the inner sum on the right-hand side of Lemma~\ref{lemma:qLidea}, and use a suitable definition of $\alpha_j$
to obtain the expansion formula:
\begin{multline}\label{expansion5}
\frac{\ellipticqrfac{aq, aq/c}{n}\ellipticpqrfac{c^2/a}{2n}{q}{p^2} }{\ellipticqrfac{c/a, c^2/a}{n}\ellipticpqrfac{aq}{2n}{q}{p^2}} 
\sum_{j=0}^{n} \frac{ \ellipticqrfac{cq^{n}, q^{-n}}{j}}{\ellipticqrfac{ cq^{-n}/a, c^2q^{n}/a} {j}}  A_{j}\\
= \sum_{k=0}^{n} 
\frac{\ellipticptheta{apq^{2k}}{p^2} \ellipticqrfac{q^{-n}, cq^{n}}{k} \ellipticpqrfac{ap, a^2q/c^2 }{k}{q}{p^2}}
{\ellipticptheta{ap}{p^2} \ellipticqrfac{aq^{1+n}, aq^{1-n}/c}{k}  \ellipticpqrfac{q, c^2p/a }{k}{q}{p^2}}q^{k}\\
\times \sum_{j=0}^k \frac{\ellipticpqrfac{ apq^{k}, q^{-k} }{j}{q}{p^2} }
{\ellipticpqrfac{c^2q^{-k}/a^2, c^2pq^{k}/a  }{j}{q}{p^2}  } A_{j}.
\end{multline}
\begin{Remark}
We observe that the summand on the left-hand side of \eqref{expansion5} contains $M_{nj}(c, c^2/aq,  q, p)$, while the inner sum on the right-hand side contains
$M_{kj}(ap, c^2p/aq,  q, p^2)$. 
\end{Remark}

\begin{Remark} Note that for each of these expansion formulas, $h_j(m)$  is a specialization of \eqref{M-terms}. The same appears on the left-hand side of the expansion formula. In addition, the inner sum on the right-hand side of the formula also contains \eqref{M-terms}. This observation is relevant to interpreting \eqref{expansion2}, \eqref{expansion3} and \eqref{expansion5} as instances of WP Bailey lemma type results. The matrix $\M$ is essentially the elliptic WP Bailey matrix. The inner sum on the right-hand side gives the interpretation as the initial WP Bailey pair; and its form on the left-hand side gives the new WP Bailey pair.


\end{Remark}


We have obtained four expansion formulas in this section, by applying Lemma~\ref{lemma:qLidea} with Warnaar's matrix inversion, and using four summation theorems. In the next section, we systematically check whether they imply any new transformation formulas.

\section{Elliptic transformation formulas}\label{sec:transformations}
The objective of  this section is to obtain transformation formulas from the expansion formulas in \S\ref{sec:expansions}. 
We choose the sequence $A_j$ in the statement of the expansion formulas so that the inner sum on the right-hand side of the expansion formula can be summed using a summation theorem. In addition to the summation theorems \eqref{10V9}, \eqref{SOW1.3}, \eqref{SOW1.4}, \eqref{lrw-1}, 
we use the following two summations, both due to Warnaar. 
We use the summation \cite[Eq.~(1.10)]{SOW2005}:
\begin{multline}\label{SOW1.10}
\sum_{j=0}^k
\frac
{\ellipticptheta{abq^{2j}}{p^2} \ellipticpqrfac{ab, aq/b, a^2q^{k+1}, q^{-k}}{j}{q}{p^2} \ellipticpqrfac{b^2}{j}{q^2}{p^2}}
{\ellipticptheta{ab}{p^2}\ellipticpqrfac{q, b^2, bq^{-k}/a, abq^{k+1}}{j}{q}{p^2} \ellipticpqrfac{a^2q^2}{j}{q^2}{p^2}}
\Big(-\frac{b}a\Big)^{j}\\
= \chi(k \text{ is even})
\frac
{\ellipticpqrfac{abq}{k}{q}{p^2}\ellipticpqrfac{q, a^2q^{2}/b^2}{k/2}{q^2}{p^2} }
{\ellipticpqrfac{aq/b}{k}{q}{p^2}\ellipticpqrfac{a^2q^2, b^2q}{k/2}{q^2}{p^2} }
.
\end{multline}
In addition, we use  \cite[Eq.~(1.15)]{SOW2005}:
\begin{equation}\label{SOW1.15}
\sum_{j=0}^k
\frac
{\ellipticptheta{bq^{2j}}{p^2} \ellipticpqrfac{b^2}{j}{q^2}{p^2}\ellipticpqrfac{c/b, bq/c, q^{k+1}, q^{-k}}{j}{q}{p^2} (-b)^{j}}
{\ellipticptheta{b}{p^2}\ellipticpqrfac{q^2}{j}{q^2}{p^2} \ellipticpqrfac{ b^2q/c, c, bq^{k+1}, bq^{-k}}{j}{q}{p^2} }
= 
\frac
{\ellipticpqrfac{bq, c/b^2}{k}{q}{p^2}\ellipticpqrfac{cq^{-k}}{k}{q^2}{p^2} }
{\ellipticpqrfac{q/b, c}{k}{q}{p^2}\ellipticpqrfac{cq^{-k}/b^2}{k}{q^2}{p^2} }
\Big( \frac{-1}{b}\Big)^k.
\end{equation}
We have stated this summation theorem without using the $_{r+1}V_r$ notation. 
\subsection*{Transformation formulas obtained from the first expansion formula}
We can choose $A_j$ in \eqref{ell-2} so that the inner sum on the right-hand side is summable using one of the six summation formulas mentioned earlier. However, three of those lead to known transformations. In addition to Frenkel and Turaev's transformation for $_{12}V_{11}$ series given in \cite[Equation~(11.2.23)]{GR90}, we obtain \cite[Theorem 4.2]{VPS2002} (see also \cite[Theorem 4.1]{SOW2003}) and \cite[Theorem 4.2]{SOW2003}. 

We begin by choosing $A_j$, so that the inner sum on the right-hand side of \eqref{ell-2} is summable by
\eqref{SOW1.3}.
This yields the transformation formula
\begin{align}\label{tr-1}
\sum_{j=0}^n &
\frac{\ellipticptheta{a^2q^{4j}}{p^2}\ellipticqrfac{bq^n, q^{-n},  d, a^2q/bcd}{j} \ellipticpqrfac{a^2, c/q}{j}{q^2}{p^2} }
{\ellipticptheta{a^2}{p^2}\ellipticqrfac{aq^{1-n}/b, aq^{1+n}, aq/d, bcd/a }{j}  \ellipticpqrfac{q^2, a^2q^3/c}{j}{q^2}{p^2}}
q^{2j} \cr
&\hspace{1.5cm}=
\frac
{\ellipticqrfac{bc/a, aq/cd, aq,  bd/a }{n}}
{\ellipticqrfac{aq/c,  bcd/a, b/a,  aq/d}{n}} 
\sum_{k=0}^n \bigg(
\frac{\ellipticptheta{a^2q^{4k}/c^2}{p^2}\ellipticpqrfac{a^2/c^2, 1/cq}{k}{q^2}{p^2}}
{\ellipticptheta{a^2/c^2}{p^2}\ellipticpqrfac{q^2, a^2q^3/c}{k}{q^2}{p^2}
}\cr
&\hspace{3cm}\times \frac
{\ellipticqrfac{bq^n , q^{-n}, d, a^2q/bcd}{k}}
{\ellipticqrfac{aq^{1-n}/bc, aq^{1+n}/c, aq/cd, bd/a}{k}}
q^{2k}\bigg).
\end{align}

To obtain \eqref{tr-1}, first
substitute $(a,b)\mapsto (a/b, a)$ in \eqref{ell-2}. Now
 take
$$A_j = \frac{\elliptictheta{-aq^{2j}}\ellipticpqrfac{a^2, b/q}{j}{q^2}{p^2}
\ellipticqrfac{q,ad/bc, aq/d}{j}}
{\elliptictheta{-a} \ellipticpqrfac{q^2, a^2q^3/b}{j}{q^2}{p^2} 
\ellipticqrfac{d,bcq/d}{j}
}\Big(\frac{cq}{a}\Big)^j.$$
The inner sum can be summed using Warnaar's summation \eqref{SOW1.3}. We relabel parameters to obtain \eqref{tr-1}.

To obtain the next transformation formula, we 
take $a=q$ and $p\mapsto p^2$ in \eqref{ell-2}, and make an appropriate choice of $A_j$.
The inner sum can be summed using Warnaar's summation \eqref{SOW1.15}. 
After relabeling parameters, and changing the nome from $p^2$ to $p$, 
we obtain the following transformation formula:
\begin{align}\label{tr-4}
\sum_{j=0}^n &
\frac{\ellipticptheta{aq^{2j}}{p}
 \ellipticpqrfac{bq^{n}, q^{-n},  c, d , aq^2/bc, q/d }{j}{q}{p} 
 \ellipticpqrfac{a^2}{j}{q^2}{p} 
 }
{\ellipticptheta{a}{p} \ellipticpqrfac{aq^{1-n}/b, aq^{1+n}, aq/c, ad,  bc/q, aq/d }{j}{q}{p}
 \ellipticpqrfac{q^2}{j}{q^2}{p} 
}
(-a)^{j}
\cr
&\hspace{1.5 cm} =
\frac{\ellipticpqrfac{ aq, b/q,  q^2/c, bc/a}{n}{q}{p}}
{\ellipticpqrfac{q^2, b/a, bc/q, aq/c}{n}{q}{p}}
\sum_{k=0}^n \bigg(\frac{\ellipticptheta{q^{1+2k}}{p}\ellipticpqrfac{adq^{-k}}{k}{q^2}{p}
}
{\ellipticptheta{q}{p} \ellipticpqrfac{dq^{-k}/a}{k}{q^2}{p}
}
\cr
&\hspace{3 cm}\times \frac{\ellipticpqrfac{ bq^{n}, q^{-n}, c, aq^2/bc, d/a}{k}{q}{p}}
{\ellipticpqrfac{ q^{2-n}/b,  q^{2+n}, q^2/c, bc/a, ad }{k}{q}{p}}
\Big(-\frac{q}{a}\Big)^{k}\bigg).
\end{align}
\begin{Remark} If we take the products in front of the sum on the right-hand side of \eqref{tr-4} to the left, and take the limit as $b\to a$,
the sum on the left-hand side reduces to the term with $j=n$. After relabelling parameters by taking $a\mapsto b/a$, and $d\mapsto ab$, we obtain a summation due to Zhao and Deng~\cite[Theorem~3.2]{ZD2016}. These authors obtained their result using matrix inversion.
\end{Remark}

We specialize $A_j$ in \eqref{ell-2} suitably,
apply \eqref{lrw-1}, and obtain the transformation formula:
\begin{align}\label{tr-5} 
\sum_{j=0}^n &
\frac{\ellipticptheta{aq^{2j}}{p} \ellipticqrfac{a, cq^{n}, q^{-n},  bq/a, abq/cd,  d}{j}  \ellipticpqrfac{a^2/b}{2j}{q}{p^2}}
{\ellipticptheta{a}{p} \ellipticqrfac{q, aq^{1-n}/c, aq^{1+n},   a^2/b, cd/b, aq/d}{j} \ellipticpqrfac{bq}{2j}{q}{p^2}}
q^{j}
\cr
&\hspace{1.5cm}=\frac{\ellipticqrfac{c/b, aq,  cd/a, bq/d}{n}}
{\ellipticqrfac{bq, c/a, aq/d, cd/b}{n}}
\sum_{k=0}^n \bigg(\frac{\ellipticptheta{bpq^{2k}}{p^2} 
\ellipticpqrfac{bp, b^2q/a^2}{k}{q}{p^2}}
{
\ellipticptheta{bp}{p^2} 
\ellipticpqrfac{q, a^2p/b}{k}{q}{p^2}
} 
\cr
&\hspace{3cm}\times
\frac
{\ellipticqrfac{cq^{n}, q^{-n},  abq/cd, d}{k}}
{\ellipticqrfac{bq^{1-n}/c, bq^{1+n}, cd/a, bq/d }{k}}
q^{k}\bigg).
\end{align}

\subsection*{Transformation formulas obtained from the second expansion formula}

If we use \eqref{10V9} to sum  the inner sum
in \eqref{expansion2}, we obtain a transformation formula due to Warnaar~\cite{SOW2003} (the second formula below Theorem 5.2).

Next, in \eqref{expansion2}, we apply \eqref{SOW1.3} with the following simultaneous substitutions:
$$(q,p, a, b)\mapsto (q^2, p^2, acq, cq/a),$$
a suitable choice of $A_j$, replace $c$ by $b$, and obtain the transformation formula:
\begin{align}\label{exp2-tr2}
\sum_{j=0}^{n} &
\frac{  \ellipticptheta{a^2b^2q^{8j+2}}{p^4} \ellipticpqrfac{b^2q^{2n}, q^{-2n}}{j}{q^2}{p^2} 
\ellipticpqrfac{a^2b^2q^2, b/aq}{j}{q^4}{p^4} 
\ellipticqrfac{-aq}{2j}}
{ \ellipticptheta{a^2b^2q^2}{p^4}  
\ellipticpqrfac{aq^{3-2n}/b,  abq^{3+2n}}{j}{q^2}{p^2} 
\ellipticpqrfac{q^4, a^3bq^{7}}{j}{q^4}{p^4} 
\ellipticqrfac{-bq}{2j}} \Big( \frac{aq^4}{b}\Big)^j \cr
&\hspace{1.5cm} = 
\frac{\elliptictheta{-b }\ellipticqrfac{-q, b/a}{n}\ellipticpqrfac{abq^3}{n}{q^2}{p^2}}
{ \elliptictheta{ - bq^{2n}} \ellipticqrfac{ -b, aq}{n}  \ellipticpqrfac{b/aq}{n}{q^2}{p^2} }
q^{-n}
\cr
&\hspace{3cm} \times
\sum_{k=0}^{n} \frac{\ellipticptheta{a^4q^{8k}}{p^4}
 \ellipticqrfac{bq^{n}, q^{-n}}{k}  \ellipticpqrfac{a^4, a/bq^{3}}{k}{q^4}{p^4} 
 } 
{ \ellipticptheta{a^4}{p^4}  \ellipticqrfac{ aq^{1-n}/b, aq^{1+n}}{k} \ellipticpqrfac{q^4, a^3bq^{7}}{k}{q^4}{p^4}
}
q^{4k}.
\end{align}

The following transformation is obtained by applying \eqref{SOW1.4} in \eqref{expansion2}:
\begin{align}\label{exp2-tr3}
\sum_{j=0}^{n} &
\frac{  \ellipticptheta{abq^{4j+1}}{p^2} \ellipticpqrfac{  b^2q^{2n}, q^{-2n}, aq/b, abq}{j}{q^2}{p^2} 
\ellipticqrfac{-aq, bq}{2j}}
{ \ellipticptheta{abq}{p^2}  
\ellipticpqrfac{ aq^{3-2n}/b,   abq^{3+2n},   b^2q^2, q^2 }{j}{q^2}{p^2} 
\ellipticqrfac{-bq, aq}{2j}} q^{2j} \cr
&\hspace{1.5cm}= \frac 
{\elliptictheta{-b }\ellipticqrfac{-q, b/a}{n}\ellipticpqrfac{abq^3}{n}{q^2}{p^2}}
{ \elliptictheta{ - bq^{2n}} \ellipticqrfac{-b, aq }{n}  \ellipticpqrfac{b/aq}{n}{q^2}{p^2} }
 q^{-n}\cr
&\hspace{3cm}
\times \sum_{k=0}^{n} \frac{\elliptictheta{-aq^{2k}}
 \ellipticqrfac{-a,  bq^{n}, q^{-n},  a/b}{k}} 
{ \elliptictheta{-a}  \ellipticqrfac{q, aq^{1-n}/b, aq^{1+n},  -bq}{k} 
}q^{k}.
\end{align}
Here we replaced $c$ by $b$ at the very end.

Next, we use \eqref{SOW1.10} with $q\mapsto q^2$, $a\mapsto a/q$ and $b\mapsto cq^2$
and a suitable choice of $A_j$ in \eqref{expansion2}, and relabel parameters, to obtain the transformation formula:
\begin{align}\label{exp2-tr4}
\sum_{j=0}^{n} &
\frac{  
\ellipticptheta{abq^{4j+1}}{p^2} 
\ellipticpqrfac{b^2q^{2n}, q^{-2n},  a/bq,  abq  }{j}{q^2}{p^2} 
 \ellipticqrfac{-aq}{2j}
\ellipticpqrfac{b^2q^4}{j}{q^4}{p^2}
}
{
\ellipticptheta{abq}{p^2}  
\ellipticpqrfac{ aq^{3-2n}/b,  abq^{3+2n}, b^2q^4, q^2 }{j}{q^2}{p^2} 
\ellipticqrfac{-bq}{2j}
 \ellipticpqrfac{a^2q^2}{j}{q^4}{p^2}
} 
 (-1)^j q^{3j} \cr
&\hspace{1.5cm}=
\frac {\elliptictheta{-b } \ellipticqrfac{-q, b/a}{n}\ellipticpqrfac{abq^3}{n}{q^2}{p^2}}
{ \elliptictheta{-bq^{2n} } \ellipticqrfac{-b, aq}{n}  \ellipticpqrfac{b/aq}{n}{q^2}{p^2} }
q^{-n}\cr
&\hspace{3cm} \times 
 \sum_{k=0}^{n} \frac{
\ellipticptheta{a^2q^{8k}}{p^2}
 \ellipticqrfac{bq^{n}, q^{-n}}{2k} 
 \ellipticpqrfac{a^2, a^2/b^2q^2}{k}{q^4}{p^2}
 } 
 {
\ellipticptheta{a^2}{p^2}
 \ellipticqrfac{ aq^{1-n}/b, aq^{1+n}}{2k} 
 \ellipticpqrfac{q^4, b^2q^6}{k}{q^4}{p^2}
 } 
q^{4k}.
\end{align}

Next, we use $q\mapsto q^2$ in \eqref{SOW1.15}, use the $a=q$, $c\mapsto b/q^2$ case of \eqref{expansion2}, and relabel parameters,
 to obtain:
\begin{align}\label{exp2-tr5}
\sum_{j=0}^{n}  &
\frac{  
\ellipticptheta{aq^{4j}}{p^2} 
\ellipticpqrfac{a^2q^{2n-4}, q^{-2n},  aq/b, bq/a}{j}{q^2}{p^2} 
 \ellipticpqrfac{-q^2}{2j}{q}{p}
\ellipticpqrfac{a^2}{j}{q^4}{p^2}
}
{\ellipticptheta{a}{p^2}  
\ellipticpqrfac{   q^{6-2n}/a, aq^{2n+2}, bq, a^2q/b }{j}{q^2}{p^2} 
\ellipticpqrfac{-a/q}{2j}{q}{p}
\ellipticpqrfac{q^4}{j}{q^4}{p^2}
} 
(-1)^jq^{3j} \cr
&\hspace{1.5cm}= \frac
{\elliptictheta{ -a/q^2}\ellipticqrfac{-q, a/q^3}{n}\ellipticpqrfac{aq^2}{n}{q^2}{p^2}}
{\elliptictheta{ -aq^{2n-2}} \ellipticqrfac{-a/q^2, q^2}{n}  \ellipticpqrfac{a/q^4}{n}{q^2}{p^2} }
q^{-n}\cr
&\hspace{3cm} \times 
\sum_{k=0}^{n} \frac{
\ellipticptheta{q^{4k+2}}{p^2}
 \ellipticqrfac{aq^{n-2}, q^{-n} }{k} 
 \ellipticpqrfac{ bq/a^2}{k}{q^2}{p^2}
 \ellipticpqrfac{ bq^{1-2k}}{k}{q^4}{p^2}
 } 
 {
\ellipticptheta{q^2}{p^2}
 \ellipticqrfac{q^{4-n}/a, q^{n+2}}{k} 
 \ellipticpqrfac{bq}{k}{q^2}{p^2}
 \ellipticpqrfac{bq^{1-2k}/a^2}{k}{q^4}{p^2}
 } 
\Big(\frac{-q^2}{a}\Big)^{k}.\cr
\end{align}

For the final formula from the second expansion formula, we use $q\mapsto q^2$, $p\mapsto p^2$, and $a\mapsto a^2$ in \eqref{lrw-1}, and use the  $c\mapsto b/aq$ case of \eqref{expansion2} with a suitable choice of $A_j$, interchange $a$ and $b$,  to obtain:
\begin{align}\label{exp2-tr6}
 \sum_{j=0}^{n} &
\frac{  \ellipticptheta{aq^{4j}}{p^2} 
\ellipticpqrfac{a, a^2q^{2n-2}/b^2, q^{-2n},  b^2q^2/a }{j}{q^2}{p^2} 
\ellipticpqrfac{a^2/b^2}{2j}{q^2}{p^4} 
 \ellipticqrfac{-bq}{2j}
}
{
\ellipticptheta{a}{p^2}
\ellipticpqrfac{q^2, b^2q^{4-2n}/a,   aq^{2n+2},   a^2/b^2}{j}{q^2}{p^2} 
\ellipticpqrfac{b^2q^2}{2j}{q^2}{p^4} 
 \ellipticqrfac{-a/b}{2j}
} 
\Big( \frac{b^2q^3}{a}\Big)^j
 \cr
&\hspace{1.5cm}= 
\frac
{\elliptictheta{-a/bq}
\ellipticqrfac{-q, a/b^2q}{n}\ellipticpqrfac{aq^2}{n}{q^2}{p^2}}
{\elliptictheta{ -aq^{2n-1}/b}
\ellipticqrfac{-a/bq, bq}{n}  \ellipticpqrfac{a/b^2q^2}{n}{q^2}{p^2} 
}
q^{-n}\cr
&\hspace{3cm}\times
\sum_{k=0}^{n} \frac{
\ellipticptheta{b^2q^{4k}p^2}{p^4}
 \ellipticqrfac{aq^{n-1}/b, q^{-n} }{k} 
 \ellipticpqrfac{b^2p^2,b^4q^2/a^2}{k}{q^2}{p^4}
 } 
 {
\ellipticptheta{b^2p^2}{p^4}
 \ellipticqrfac{ b^2q^{2-n}/a, bq^{n+1}}{k} 
 \ellipticpqrfac{q^2, a^2p^2/b^2}{k}{q^2}{p^4}
 } 
q^{2k}.
\end{align}

\begin{Remarks} \

\begin{enumerate}
\item There are two sets of six transformation formulas that appear by considerations as above.
\begin{enumerate}
\item The first set of transformation formulas is obtained from 
 the first expansion formula. These are: Frenkel and Turaev's transformation for $_{12}V_{11}$  (\cite[Equation~(11.2.23)]{GR90}), Spiridonov's \cite[Theorem 4.2]{VPS2002} (that is, \cite[Theorem 4.1]{SOW2003}), Warnaar's transformation \cite[Theorem 4.2]{SOW2003}, \eqref{tr-1}, \eqref{tr-4} and \eqref{tr-5}.  In terms of the WP Bailey lemma,  they are consequences of Spiridonov~\cite[Theorem~4.3]{VPS2002}. 

 \item The second set is obtained from the second expansion formula. These include Warnaar's transformation given in \cite{SOW2003} (the second formula below Theorem 5.2), and 
\eqref{exp2-tr2}, \eqref{exp2-tr3}, \eqref{exp2-tr4}, \eqref{exp2-tr5}, and, \eqref{exp2-tr6}. Equivalently, these can be obtained from Warnaar~\cite[Theorem 5.2]{SOW2003}, 
a WP Bailey lemma type result. 
\end{enumerate}
\item If we use any of the aforementioned summation theorems to sum the inner sum of the third and fourth expansion formulas (\eqref{expansion3} and  \eqref{expansion5}, respectively), we obtain special cases of the first set of transformation formulas.
\item As we have seen, \eqref{tr-4} contains a summation theorem from \cite{ZD2016}. 
We have not listed the summation theorems implied by these transformation formulas. 
\item The transformation formulas can be plugged back into the expansion formulas of \S\ref{sec:expansions} to obtain more formulas; this corresponds to following steps along the branches of the WP Bailey tree. 

 \end{enumerate}
 \end{Remarks}

The calculations in \S\ref{sec:expansions} illustrate how, beginning with a summation theorem, we can try to derive an expansion formula by using Lemma~\ref{lemma:qLidea}. If the entries of $\M$ are present in the summation, then we can use the matrix inversion of Proposition~\ref{warnaar-matrix}. The resulting expansion formula is a WP Bailey lemma type of result. Subsequently, the expansion formula can be combined with (possibly other) summations to obtain transformation formulas.
This has resulted in eight new transformation formulas in this section, and thus appears to be quite productive.



\section{Remarks on $q$-hypergeometric special cases}\label{sec:special}
We can take the nome $p=0$ in our formulas to obtain results for basic hypergeometric series. 
Taking $p=0$ is immediate; the $q, p$-shifted factorials reduce to $q$-rising factorials. 
Note that several results obtained in this fashion are for bibasic series. Some of Warnaar's summations 
that we have used here  extend results of Nassrallah and Rahman (see \cite[\S 3.10]{GR90}). But we have been unable to find the $p=0$ cases of the transformation formulas in \S\ref{sec:transformations} in the literature; they appear to be new.

As one example, note that the $p=0$ case of \eqref{tr-5} can be written as
\begin{multline}\label{special-tr-5} 
\sum_{j=0}^n 
\frac{\big(1-{aq^{2j}}\big) \qrfac{a, q^{-n}, aq/b, c, d,   \sqrt{b},- \sqrt{b}, \sqrt{bq}, -\sqrt{bq}, a^3q^{1+n}/bcd  }{j} }
{(1-a)\qrfac{q, aq^{1+n}, b,  aq/c, aq/d, aq/\sqrt{b},-aq/\sqrt{b},a\sqrt{q/b},- a\sqrt{q/b}, bcdq^{-n}/a^2  }{j} }
q^{j}
\cr
=\frac{\qrfac{aq, a^2q/bc,a^2q/ bd, aq/cd}{n}}
{\qrfac{a^2q/b, aq/c,   aq/d, a^2q/bcd }{n}}
\sum_{k=0}^n 
\frac
{\qrfac{q^{-n}, c,  d , a^3q^{1+n}/bcd,  a^2q/b^2}{k}}
{\qrfac{a^2q^{1+n}/b, a^2q/bc, a^2q/bd,  cdq^{-n}/a, q}{k}}
q^{k}.
\end{multline}

 The sum on the left-hand side is a very-well-poised, and balanced, $_{12}\phi_{11}$; on the right-hand side the sum is a $_5\phi_4$, which is nearly-poised  of the second kind (see \cite{GR90}, pages 4 and 39, for this terminology). This transformation formula is a companion to  a formula of Bailey, given in \cite[Equation (2.8.3)]{GR90}, as well as  two of Andrews and Berkovich~\cite[Equations (3.7) and (3.8)]{AB2002}. 
 
 When the parameters $c$ and $d$ are taken to be $\pm aq/\sqrt{b}$, the $_{12}\phi_{11}$ reduces to an $_8\phi_7$, and can be summed using Jackson's sum \cite[Equation~(2.6.2)]{GR90}, to yield a summation 
 theorem: 
\begin{equation}\label{special-tr-5a} 
\sum_{k=0}^n 
\frac
{\qrfac{q^{-n},   aq^{n-1}, b^2q}{k} \pqrfac{abq^2}{k}{q^2}}
{\qrfac{abq^{1+n},   bq^{2-n}, q}{k}\pqrfac{ab}{k}{q^2}}
q^{k}
=\frac{\qrfac{ abq, b}{n} \pqrfac{ a/bq, a/b}{n}{q^2}}
{\qrfac{a/b, 1/bq}{n} \pqrfac{ab, abq}{n}{q^2}}.
\end{equation}
Similarly, we take $c=-aq/\sqrt{b}$ and $d=a\sqrt{q/b}$ in \eqref{special-tr-5}, use Jackson's sum~\cite[Equation~(2.6.2)]{GR90}, 
and replace $q$ by $q^2$, and $b$ by $b^2$, to obtain
\begin{multline}\label{special-tr-5-b} 
\sum_{k=0}^n 
\frac
{\pqrfac{q^{-2n}, -aq^2/b, a^2q^2/b^4,  -aq^{2n-1}, aq/b}{k}{q^2} }
{\pqrfac{a^2q^{2n+2}/b^2,   -aq^{3-2n}/b^2, -a/b, aq/b, q^2}{k}{q^2}}
q^{2k} 
=\frac{\pqrfac{  -aq/b^2, a^2q^2/b^2}{n}{q^2}\qrfac{b, -b/q}{2n}}
{\pqrfac{b^2,  -b^2/aq}{n}{q^2}\qrfac{aq/b, -a/b}{2n}}.
\end{multline}
Again, we have been unable to find these in the literature; they are similar to analogous special cases of Bailey's transformation mentioned above.

These identities suggest that our technique should be useful in the study of basic hypergeometric series. 

In conclusion, we can say that given a summation theorem, we can attempt to combine it with a known matrix inversion and derive an expansion formula by using Lemma~\ref{lemma:qLidea}. The expansion formula thus obtained  is a Bailey lemma type result. It can be combined with (possibly other) summations to obtain transformation formulas. 

\subsection*{Acknowledgements} We thank Michael Schlosser for helpful remarks, Monu Jangra for assistance in checking results on Sagemath, and the anonymous referees for several astute observations and helpful suggestions.



\end{document}